\documentclass[11pt,reqno]{amsart}

\usepackage{amscd,hyperref}
\usepackage{amsmath, amssymb}
\usepackage{amsfonts}

\newcommand{\de}{\partial}

\newcommand{\ov}[1]{\bar{#1}}

\newcommand{\tr}[2]{\mathrm{tr}_{#1}{#2}}
\newcommand{\ti}[1]{\tilde{#1}}

\newcommand{\oNab}{\mkern 1mu\overline{\mkern-0.5mu \nabla\mkern-1mu}\mkern 1mu}

\renewcommand{\leq}{\leqslant}
\renewcommand{\geq}{\geqslant}

\newcommand{\N}{\mathbb{N}}
\newcommand{\Z}{\mathbb{Z}}

\newcommand{\R}{\mathbb{R}}
\newcommand{\C}{\mathbb{C}}

\newtheorem{theorem}{Theorem}[section]

\newtheorem{proposition}[theorem]{Proposition}

\numberwithin{equation}{section}

\theoremstyle{definition}
\newtheorem{rk}[theorem]{Remark}

\begin{document}

\title{Remarks on the collapsing of \\ torus fibered Calabi-Yau manifolds}
\author{Hans-Joachim Hein}
\address{Department of Mathematics, University of Maryland, College Park, MD 20742, USA}
\email{hein@umd.edu}
\author{Valentino Tosatti}
\address{Department of Mathematics, Northwestern University, Evanston, IL 60208, USA}
\email{tosatti@math.northwestern.edu}
\date{\today}
\thanks{The first-named author is supported in part by NSF grant DMS-1514709. The second-named author is supported in part by a Sloan Research Fellowship and NSF grant DMS-1308988.}
\begin{abstract}
One of the main results of the paper \cite{GTZ} by Gross-Tosatti-Zhang establishes estimates on the collapsing of Ricci-flat K\"ahler metrics on holomorphic torus fibrations. We remove a projectivity assumption from these estimates and simplify some of the underlying analysis.
\end{abstract}
\maketitle
\markboth{Hans-Joachim Hein and Valentino Tosatti}{Remarks on the collapsing of torus fibered Calabi-Yau manifolds}

\thispagestyle{empty}

\section{Introduction}\label{intro}

Let $f:M\to N$ be a surjective holomorphic map between compact K\"ahler manifolds. We assume that $M$ is Calabi-Yau, i.e.~$c_1(M)_\R =0$, and that all  smooth fibers of $f$ are complex $n$-tori, i.e.~$M_y =$ $f^{-1}(y) =\mathbb{C}^n/\Lambda_y$ for all $y\in N\setminus f(S)$, where $S$ is the set of critical points of $f$.
Fix K\"ahler metrics $\omega_M,\omega_N$ on $M,N$, put $\omega_0=f^*\omega_N$, and for all $t \in (0,1]$ let $\ti{\omega}_t$ be the unique Ricci-flat
K\"ahler metric on $M$ cohomologous to $\omega_0+t\omega_M$, whose existence is guaranteed by Yau's theorem \cite{Ya}.

The goal of this note is to prove the following estimate for $\ti{\omega}_t$:

\begin{theorem}\label{main}
Given any compact set $K\subset M\setminus f^{-1}(f(S))$ and any $k \in \N_0$, there exists a constant $C_{K,k} < \infty$, which does not depend on $t$, such that
\begin{equation}\label{estimate}
\|\ti{\omega}_t\|_{C^k(K,\omega_M)}\leq C_{K,k}
\end{equation}
holds uniformly for all $t \in (0,1]$.
\end{theorem}

Theorem \ref{main} was proved in \cite{GTZ} (see Proposition 4.6 there) assuming that $M$ is projective algebraic, and our first purpose here is precisely to remove this assumption. Our second purpose is to simplify the analysis in \cite{GTZ} by applying a local Calabi-Yau type $C^3$ estimate in place of Evans-Krylov theory, so that the somewhat delicate $i\partial\bar\partial$-lemma of \cite[Proposition 3.1]{GTZ} is no longer needed.

The proofs of the first two main theorems of \cite{GTZ} (i.e.~Theorems 1.1 and 1.2 there) did not rely on the
assumption that $M$ is projective except through the proof of (\ref{estimate}). Thus, we are also removing the
projectivity assumption from those two results. Similar remarks apply to \cite[Section 5]{tsz}.

Let us now quickly sketch the contents of this note. We fix a small coordinate ball $B \subset N\setminus f(S)$ and call $U=f^{-1}(B)$. Recall that a closed real $(1,1)$-form on $U$ is called {\em semi-flat} if it restricts to a
flat K\"ahler metric on every fiber $f^{-1}(y)$ ($y\in B$). Let $p: B \times \C^n \to U$ be the universal holomorphic cover of $U$. The key result to deriving (\ref{estimate}) is the following:

\begin{theorem}\label{sf}
There exists a semi-flat form $\omega_{\rm SF}\geq 0$ on $U$ such that $p^*\omega_{\rm SF}=i\partial\bar\partial\eta$ for a smooth real-valued function $\eta$ on $B\times\mathbb{C}^n$ with the scaling property
\begin{equation}\label{scale}
\eta(y,\lambda z)=\lambda^2\eta(y,z)\;\,(\lambda \in \R).
\end{equation}
\end{theorem}

This was proved in \cite[Section 3]{GTZ}, under the assumption that $M$ is projective, using the fact that polarized abelian varieties are classified by
the Siegel upper half-space. Our observation here, which lets us remove the projectivity condition, is that the Siegel upper half-space more generally
classifies (marked) polarized \emph{complex tori}: pairs consisting of a complex $n$-torus $T$ and a group isomorphism
$\phi:\Lambda^2\mathbb{Z}^{2n} \to$ $H^2(T,\mathbb{Z})$ such that $\phi_{\mathbb{R}}(x)$ is a K\"ahler class on $T$, where
$x\in \Lambda^2\mathbb{R}^{2n}$ is given and fixed but not necessarily rational. This is implicit in papers of Fujiki \cite[Proposition 14]{Fu} and Schumacher \cite[Theorem 4.4]{Sc}, and we will make it more explicit in the proof of Theorem \ref{sf} in Section \ref{s:sf}.

Section \ref{s:main} then completes the proof of Theorem \ref{main} along the lines indicated above.

\section{Construction of the semi-flat form}\label{s:sf}

We will show Theorem \ref{sf} by writing down an explicit formula for $\eta$ on $B \times \C^n$ and checking that
$i\partial\bar\partial\eta$ is invariant under the automorphism group of the covering $p$. In fact, the formula for $\eta$ takes the following form, from which (\ref{scale}) and the semi-flat property are clear:
\begin{equation}\label{eta}
\eta(y,z) = -\frac{1}{4}\sum_{\ell,m=1}^n ({\rm Im}\,Z(y))^{-1}_{\ell m}(z_\ell - \bar{z}_\ell)(z_m - \bar{z}_m).
\end{equation}
Here $Z$ denotes an appropriately constructed \emph{period map} from $B$ to the Siegel upper half-space $\mathfrak{H}_n$ of symmetric $n \times n$ complex matrices with positive definite imaginary parts. It remains to explain the construction of $Z$, which will be done in two steps and represents the main difference between this section and \cite[Section 3]{GTZ}, and to check translation invariance and semi-positivity of $i\partial\bar\partial\eta$.

\subsection{Construction of a polarization} Fix a basis $(v_1(y),\dots,v_{2n}(y))$ of the lattice $\Lambda_y$ that varies holomorphically with $y \in B$, and let $(\xi^1(y),\dots,\xi^{2n}(y))$ be
the $\R$-dual basis of $1$-forms on $\C^n$. Then the classes $[\xi^i(y)\wedge\xi^j(y)]\in H^2(M_y,\mathbb{Z}) \subset H^2(M_y, \R)$ with $i < j$ form a basis of $H^2(M_y,\mathbb{Z})$.

\begin{proposition}
Let $\Omega$ be a real $2$-form on $U$ whose restriction to $M_y$ is closed for all $y \in B$, and expand $[\Omega|_{M_y}] = \sum_{i<j} P_{ij}(y) [\xi^i(y)\wedge\xi^j(y)]$. If $\Omega$ is closed, then the $P_{ij}(y)$ do not depend on $y$.
\end{proposition}

\begin{proof}
We use the Gau{ss}-Manin connection $\nabla^{\rm GM}$ on the smooth $\R$-vector bundle $R^2 f_* \mathbb{R}\otimes\mathcal{C}^\infty_B$. By definition, the sections $[\xi^i(y)\wedge\xi^j(y)]$
of this vector bundle form a basis of the space of $\nabla^{\rm GM}$-parallel sections. On the other hand, since $\Omega$ is closed, Cartan's magic formula yields that
$\nabla^{\rm GM}[\Omega|_{M_y}]=0$; see \cite[Corollary 4.4.4]{CMP} or \cite[Proposition 9.14]{Vo}. This immediately implies the claim.
\end{proof}

We apply this to the K\"ahler form $\Omega=\omega_M$. The only difference with the projective case considered in \cite{GTZ} is that there, we could have chosen the $P_{ij}(y)$ to be $\Z$-valued to begin with and hence \emph{trivially} independent of $y$. Now let $Q \in \R^{2n \times 2n}$ denote the unique skew-symmetric matrix with $Q_{ij} = P_{ij}$ for all $i < j$ and define, for each $y \in B$, a closed real $2$-form $\omega_y$ on the torus $M_y$ by setting
\begin{equation}
\label{pola}
\omega_y =\frac{1}{2}\sum Q_{ij} \xi^i(y) \wedge\xi^j(y).
\end{equation}
Notice that $\omega_y$ is translation-invariant and cohomologous to $\omega_M|_{M_y}$. Thus $\omega_y$ is itself a K\"ahler form
and indeed the unique flat K\"ahler form on $M_y$ cohomologous to $\omega_M|_{M_y}$.

\begin{rk}
An interesting class of complex $n$-torus bundles whose total spaces do not admit \emph{any} K\"ahler forms was introduced by Atiyah \cite{At}. The base of all of these bundles is the projective variety ${\rm SO}(2n)/{\rm U}(n)$ parametrizing linear complex structures on $\C^n$ compatible with the Euclidean metric and standard orientation. For $n = 2$ they are precisely the twistor spaces of abelian surfaces.

In fact, any choice of a polarization $x \in \Lambda^2\R^{2n}$ as at the end of Section \ref{intro} defines an embedding $F_x$ of the Stein manifold $\mathfrak{H}_n$ into the moduli space $\mathfrak{M}_n = {\rm GL}(2n,\R)/{\rm GL}(n,\C)$ of {all} complex $n$-tori, and for each $T \in F_x(\mathfrak{H}_n)$ there exists a natural subvariety $X_{T,x} \subset \mathfrak{M}_n$,   isomorphic to ${\rm SO}(2n)/{\rm U}(n)$ and intersecting $F_x(\mathfrak{H}_n)$ transversely in $T$, which serves as the base of an Atiyah bundle.
\end{rk}

\subsection{Construction of the period map}
From here on, we can follow the arguments in \cite[p.~371]{He}\footnote{The discussion in \cite{He} involves a certain $2$-form $\omega$ on $U$ (not assumed to be $(1,1)$ or closed) but in fact only depends on the family of fiberwise $2$-forms $\{\omega|_{M_y}\}_{y \in B}$. In our case this family is defined by (\ref{pola}).} to construct a holomorphic period map $Z: B \to \mathfrak{H}_n$. We will give some more details than in \cite{He}.

We begin by setting $T = (v_1, \dots, v_{2n}) \in \mathcal{O}(B,\C^{n \times 2n})$. Also, we choose $S \in {\rm GL}(2n,\R)$ such that $S^{\rm tr}Q S = (\begin{smallmatrix} 0 & 1 \\ -1 & 0\end{smallmatrix})$ in terms of the canonical decomposition $\R^{2n} = \R^n \oplus \R^n$ (here and in the following we denote by $A^{\rm tr}$ the transpose of a square matrix $A$). Observe that $S$ is unique only up to right
multiplication by Sp$(2n,\R)$ and that the period map $Z$ to be defined momentarily will depend on this choice, but only up to the action of a holomorphic isometry of $\mathfrak{H}_n$. In the course of the proof of Proposition \ref{siegel}, we will see that the first $n$ columns of $TS$ are $\C$-linearly independent. We can therefore write $TS = R(1, Z)$ with $R \in \mathcal{O}(B,{\rm GL}(n,\C))$ and $Z \in \mathcal{O}(B,\C^{n \times n})$.

\begin{rk}\label{basis}
We have an additional freedom of choosing a lattice basis, i.e.~of replacing $T$ by $TA$ for some constant matrix $A \in {\rm GL}(2n,\Z)$. Then $Q$ changes to $A^{\rm tr}QA$, but $S$ becomes $A^{-1}S$ so that the period map $Z$ remains unchanged. We will not make use of this freedom here, but if $Q$ is integral as in \cite{GTZ}, then we could arrange in this way that
$Q = (\begin{smallmatrix} 0 & \Delta \\ -\Delta & 0\end{smallmatrix})$, where $\Delta = {\rm diag}(d_1, \dots, d_n)$ for some positive integers $d_1|d_2|\dots|d_n$; compare \cite[p.~304, Lemma]{GH}.
\end{rk}

We now work at a given point $y \in B$ and for simplicity write $\omega = \omega_y$ and $Z = Z(y)$.
Recall that, even though this was not immediate from (\ref{pola}), $\omega = \omega_y$ is indeed a K\"ahler form on $M_y$.

\begin{proposition}\label{siegel}
{\rm (a)} It holds that $Z \in \mathfrak{H}_n$. In fact, this is equivalent to $\omega$ being positive $(1,1)$.

{\rm (b)} Moreover, $\omega = i \sum H_{\ell m} dz_\ell \wedge d\bar{z}_m$, where $H^{-1} = 2\bar{R}({\rm Im}\,Z)R^{\rm tr} = i\bar{T}Q^{-1}T^{\rm tr}$.
\end{proposition}
\begin{proof}
We view the columns of $TS$ as an $\R$-basis of $\C^n$ and denote the $\R$-dual basis by $\zeta^1$, \dots, $\zeta^{2n}$.
Then $\xi^i = \sum S_{ij}\zeta^j$ and hence $\omega = \sum_{k=1}^n \zeta^k \wedge \zeta^{n+k}$. As an aside, we can now deduce that the first $n$ columns of $TS$ are $\C$-linearly independent, thus justifying the construction of $R$ and $Z$. Indeed, let $V$ denote their $\R$-span in $\C^n$. Then $\omega(v,iv) = 0$ for all $v \in V \cap iV$ by the
preceding formula; since $\omega$ is positive $(1,1)$, this means that the $\C$-subspace $V \cap iV$ is trivial, as desired.

The matrix $R$ defines a
$\C$-isomorphism $R: \C^n \to \C^n$, so $\omega$ is positive $(1,1)$ if and only if $R^*\omega$ is.
But $R^*\omega = \sum_{k=1}^n \eta^k \wedge \eta^{n+k}$ for the basis of $1$-forms $\eta^1, \dots, \eta^{2n}$ that is
$\R$-dual to the basis of column vectors of $(1,Z)$. To understand the condition for this form to be positive $(1,1)$, we define matrices $A,B,\tilde{A}, \tilde{B} \in \C^{n \times n}$ by
$\eta^k = \sum A_{k\ell} dz_\ell + B_{k\ell}d\bar{z}_\ell$ and $\eta^{n+k} = \sum \tilde{A}_{k\ell}dz_\ell + \tilde{B}_{k\ell}d\bar{z}_\ell$ for
$k,\ell \in \{1,\dots,n\}$. Denoting the skew-symmetric part of a matrix by a superscript ``skew'', we then have that
$$R^*\omega =\sum (A^{\rm tr}\tilde{A})^{\rm skew}_{\ell m} dz_{\ell} \wedge dz_m + (A^{\rm tr}\tilde{B} - \tilde{A}^{\rm tr}B)_{\ell m}dz_\ell \wedge d\bar{z}_m +  (B^{\rm tr}\tilde{B})^{\rm skew}_{\ell m}d\bar{z}_\ell \wedge d\bar{z}_m,$$
so $R^*\omega$ is positive $(1,1)$ if and only if $(A^{\rm tr}\tilde{A})^{\rm skew} =  (B^{\rm tr}\tilde{B})^{\rm skew} = 0$ and
$\frac{1}{i}(A^{\rm tr}\tilde{B} - \tilde{A}^{\rm tr}B)$ is Hermitian positive definite. In order to rephrase this in terms of $Z$, observe that we have
$$A + B = 1, \;\; AZ + B\bar{Z} = 0, \;\; \tilde{A} + \tilde{B} = 0, \;\;\tilde{A}Z + \tilde{B}\bar{Z} = 1.$$
This yields that $\tilde{A}(Z - \bar{Z}) = 1$ (so, in particular, both factors are invertible) and $A = -\bar{Z}\tilde{A}$. Given this, $(A^{\rm tr}\tilde{A})^{\rm skew} = 0$ is clearly equivalent to $Z$ being symmetric, and then $\frac{1}{i}(A^{\rm tr}\tilde{B} - \tilde{A}^{\rm tr}B) = i \tilde{A}^{\rm tr} = \frac{1}{2}({\rm Im}\,Z)^{-1} > 0$. Thus, $Z \in \mathfrak{H}_n$. Also, $R^*\omega = \frac{i}{2}\sum ({\rm Im}\,Z)^{-1}_{\ell m} dz_{\ell} \wedge d\bar{z}_m$, which implies (b).
\end{proof}

\subsection{Comparison with \cite{GTZ}} Let us assume for the moment that $Q$ is integral as in \cite{GTZ}. We wish to compare our formalism, specialized to this case, with the treatment in \cite{GTZ}.

We can assume without loss that $Q$ is as in Remark \ref{basis}. Also, let us write $\Delta = \Sigma^2$, where $\Sigma > 0$.
Then we choose $S = {\rm diag}(\Sigma^{-1}, \Sigma^{-1})$, so that $T = (R\Sigma, RZ\Sigma) $. Finally, we apply the automorphism $R^{-1}$ to simplify the picture, mapping $T$ to
$(\Sigma, Z\Sigma)$ and $\omega$ to $\frac{i}{2}\sum ({\rm Im}\,Z)^{-1}_{\ell m} dz_{\ell} \wedge d\bar{z}_m$.

On the other hand, \cite[p.~528]{GTZ} states that we can assume that
$T = (\Delta, Z_0)$ with $Z_0 \in \mathfrak{H}_n$ and that $\omega = \frac{i}{2}\sum ({\rm Im}\,Z_0)^{-1}_{\ell m} dz_{\ell} \wedge d\bar{z}_m$.
More precisely, this means that, given the original $T = (v_1, \dots, v_{2n})$, there exists $R_0 \in {\rm GL}(n,\C)$
such that $T = R_0(\Delta, Z_0)$ and such that $R_0^*\omega$ equals the above. We then recover the previous picture by setting
$R = R_0\Sigma$ and $Z = \Sigma^{-1}Z_0\Sigma^{-1}$.

\subsection{Translation invariance and semi-positivity} We now return to the general case. Again we apply
$R^{-1}$ to simplify matters, so that now $T = (1,Z)S^{-1}$ and $\omega = \frac{i}{2}\sum ({\rm Im}\,Z)^{-1}_{\ell m} dz_{\ell} \wedge d\bar{z}_m$.

It remains to check that, for $\eta$ as in (\ref{eta}), the form $i\partial\bar\partial\eta$ is nonnegative and invariant under the deck transformations of the covering $p: B \times \C^n \to U$. (Of course the resulting semi-flat form on $U$ will then satisfy $\omega_{\rm SF}|_{M_y} = \omega_y$ for all $y \in B$, but we do not need this for Theorem \ref{sf}.)

As in \cite{GTZ}, the key is to prove that $i\partial\bar\partial\eta$ is invariant under translation by \emph{arbitrary} flat sections of the Gau{ss}-Manin connection on $R^1 f_*\R \otimes \mathcal{C}^\infty_B$. Concretely, we must show that if $T = (v_1, \dots, v_{2n})$, then for all functions $\sigma: B \to \C^n$ of the form $\sigma(y) = \sum \lambda_i v_i(y)$ with constants $\lambda_i \in \R$, the difference
$(\eta \circ T_\sigma )- \eta$ is pluriharmonic on $B \times \C^n$, where $T_\sigma(y,z) = (y, z + \sigma(y))$. Indeed, once we have this, then specializing to $\lambda_i \in \Z$ yields the desired invariance of $i\partial\bar\partial\eta$ under the deck group; moreover, it then suffices to check that $i\partial\bar\partial\eta \geq 0$ \emph{at the zero section} $z = 0$, which is clear from (\ref{eta}) (the vertical components are given by $\omega$, and the horizontal and mixed ones vanish; compare \cite[p.~529]{GTZ}).

The proof of the required translation property is similar to the one in \cite[p.~529]{GTZ}. Indeed, writing $S^{-1} = (\begin{smallmatrix} A & B \\ C & D\end{smallmatrix})$ with $A,B,C,D \in \R^{n \times n}$, a straightforward computation shows that
\begin{align*}
\begin{split}
(\eta \circ T_\sigma) - \eta =  \sum_{j,\ell=1}^n (\lambda_j C_{\ell j} + \lambda_{n+j} D_{\ell j})(2{\rm Im}\, z_\ell + \sum_{k = 1}^n (\lambda_k (({\rm Im}\,Z)C)_{\ell k} + \lambda_{n+k} (({\rm Im}\,Z)D)_{\ell k})),
\end{split}
\end{align*}
which is obviously pluriharmonic. This completes the proof of Theorem \ref{sf}.

\begin{rk}
In fact, we have proved a more precise result: given any K\"ahler form $\omega_M$ on $M$, there exists a semi-flat form $\omega_{\rm SF} \geq 0$ on $U$ as in Theorem \ref{sf} such that $\omega_{\rm SF}|_{M_y}$ is cohomologous to $\omega_M|_{M_y}$ for all $y \in B$. But this is not required in any of the following. In particular, unlike in \cite{GTZ}, there will be no need to construct $\omega_{\rm SF}$ using the same $\omega_{M}$ that was used to construct $\tilde{\omega}_t$.
\end{rk}

\section{Main estimates}\label{s:main}
The proof of Theorem \ref{main} given in \cite{GTZ}, under the assumption that $M$ is projective, has two parts: (a) the $C^0$ estimate of $\tilde{\omega}_t$, and (b) the $C^k$ estimate of $\ti{\omega}_t$ for $k \geq 1$. The projectivity of $M$ was used in the course of the construction of $\omega_{\rm SF}$ in (a), and in (b) through the proof of a certain $i\partial\bar\partial$-lemma for abelian fibrations \cite[Proposition 3.1]{GTZ} analogous to \cite[Lemma 4.3]{GW} and \cite[Proposition 3.7]{He}.

Thanks to Theorem \ref{sf}, (a) now goes through verbatim without projectivity.
Using the work in Section \ref{s:sf}, the $i\partial\bar\partial$-lemma could easily be extended to the general K\"ahler setting as well, and then the rest of (b) would go through without changes. However, we will give a simpler and more robust argument for (b) here by working at the level of metrics rather than potentials, employing a local version of Yau's
$C^3$ estimate. This eliminates the rather difficult $i\partial\bar\partial$-lemma from the proof.

Throughout this section, we will assume that the compact set $K$ of Theorem \ref{main} is so small that it can be identified with one of its preimages under the universal covering map $p: B \times \C^n \to U$ for some sufficiently small coordinate ball $B \subset N \setminus f(S)$, where $U = f^{-1}(B)$.

\subsection{The $C^0$ estimate of $\tilde{\omega}_t$} Let $\omega_{\rm SF} \geq 0$ be the semi-flat form on $U$ constructed in Theorem \ref{sf}, so
that $\omega_0+\omega_{\rm SF}$ is a semi-flat K\"ahler form on $U$.
As in \cite{GTZ} define
$\lambda_t:B\times\mathbb{C}^n\to B\times\mathbb{C}^n$ by
$$\lambda_t(y,z)=\left(y,\frac{z}{\sqrt{t}}\right).$$
Given Theorem \ref{sf}, the same proof as in \cite[Lemma 4.2]{GTZ} (which relies on some estimates from \cite{To}) shows that there exists a constant $C_B$ such that
\begin{equation*}
C_B^{-1}p^*(\omega_0+\omega_{\rm SF})\leq \lambda_t^*p^*\ti{\omega}_t\leq C_B p^*(\omega_0+\omega_{\rm SF})
\end{equation*}
holds on $B\times\mathbb{C}^n$, uniformly in $t \in (0,1]$. This trivially implies that for each compact set $K \subset B \times \C^n$ there exists a constant $C_K$ such that, with $\delta$ denoting the Euclidean metric on $B\times\mathbb{C}^n$,
\begin{equation}\label{metric2}
C^{-1}_{K}\delta\leq \lambda_t^*p^*\ti{\omega}_t\leq C_{K}\delta.
\end{equation}
In fact, $C_K$ only depends on $B \subset N \setminus f(S)$ if we take $K$ to lie in a fixed fundamental domain of the deck group action, which we can. Note that (\ref{metric2}) implies (\ref{estimate}) for $k = 0$ because $t \leq 1$.

\subsection{The $C^1$ estimate of $\tilde{\omega}_t$}
The following is a simple modification of an argument of
Sherman-Weinkove \cite{SW} for the K\"ahler-Ricci flow. For simplicity of notation, we denote by $g$ the Riemannian metric associated with the K\"ahler form $\lambda_t^*p^*\ti{\omega}_t$, but of course $g$ still depends on $t$.

\begin{proposition}\label{calabi}
Given any compact set $K\subset B\times\mathbb{C}^n$, there exists a constant $C_K$ such that
\begin{equation}\label{great}
\sup\nolimits_K |\nabla g|^2 \leq C_K
\end{equation}
holds uniformly for all $t \in (0,1]$, where the connection and norm are the ones associated with $\delta$.
\end{proposition}

\begin{proof}
Fix a slightly larger compact set $K'$ containing $K$ in its interior. On $K'$, by (\ref{metric2}),
\begin{equation}\label{assum}
C^{-1}\delta\leq g\leq C\delta
\end{equation}
for some generic constant $C = C_{K'} = C_{K}$ independent of $t$. Let $\psi \geq 0$ be a smooth cutoff function supported
in $K'$ with $\psi\equiv 1$ on $K$ such that
$|\nabla \psi|^2\leq C$ and $\Delta(\psi^2) \geq -C$, where the gradient and Laplacian are the ones associated with $\delta$.
Then, by \eqref{assum}, $|\nabla_g \psi|_g^2\leq C$ and $\Delta_g(\psi^2) \geq -C$.

Following Yau \cite{Ya}, we define
$$S=|\nabla g|_g^2=g^{i\ov{\ell}}  g^{j\ov{q}} g^{p\ov{k}} \de_i g_{j\ov{k}} \de_{\ov{\ell}}g_{p\ov{q}}.$$
The norm (but not the connection) is now the one associated with $g$, which makes no difference for the final estimate because of (\ref{assum}). If $T$ denotes the difference of the Christoffel symbols of $g$ and $\delta$, restricted to the $(1,0)$-tangent bundle, then $T$ is a
tensor and it is easy to see that $S=|T|^2_g$.

Using that $g$ is Ricci-flat and $\delta$ is flat, the version of the Calabi-Yau $C^3$ estimate in \cite{PSS} gives
$$\Delta_g S = |\nabla_g T|^2_g+|\oNab_g T|^2_g.$$
Notice here that $T$ is not a real-valued tensor.
We can then compute that
\begin{align*}
\Delta_g (\psi^2 S)\geq \psi^2(|\nabla_g T|^2_g+|\oNab_g T|^2_g)-C S-2|\langle \nabla_g \psi^2,\nabla_g S\rangle_g|
\geq -CS,
\end{align*}
where we have used that, by Young's inequality,
\begin{align*}
2|\langle \nabla_g \psi^2,\nabla_g S\rangle_g|
=4\psi|\langle \nabla_g \psi, \nabla_g |T|^2_g\rangle_g|\leq C \psi |\nabla_g|T|^2_g|_g
\leq \psi^2(|\nabla_g T|^2_g+|\oNab_g T|^2_g)+CS.
\end{align*}
On the other hand, the Aubin-Yau $C^2$ estimate, using again that $g$ is Ricci-flat and $\delta$ is flat, gives
$$\Delta_g \tr{\delta}{g}=\delta^{i\ov{\ell}} g^{j\ov{q}} g^{p\ov{k}} \de_i g_{j\ov{k}} \de_{\ov{\ell}}g_{p\ov{q}}\geq C^{-1}S,$$
using \eqref{assum}. Thus, if we pick $C'$ large enough depending on the value of $C$ up to here, then
\begin{equation*}
\Delta_g (\psi^2 S+C'\tr{\delta}{g})\geq 0.
\end{equation*}
Hence the maximum of $\psi^2 S+C'\tr{\delta}{g}$ in $K'$ is achieved on the boundary of $K'$, which implies that
$\sup_{K}S\leq \sup_{K'}(\psi^2 S+C'\tr{\delta}{g})\leq C'\sup_{\partial K'} \tr{\delta}{g} \leq C$ as required, using (\ref{assum}).
\end{proof}

Now \eqref{great} indeed implies \eqref{estimate} for $k=1$, again because $t \leq 1$; compare \cite[Lemma 4.5]{GTZ}.

\subsection{Higher order estimates}
To prove (\ref{estimate}) for $k \geq 2$, we use a standard bootstrap argument. Since $g$ is Ricci-flat K\"ahler, we have that $\partial_i\partial_{\ov{\jmath}}\log \det(g_{k\ov{\ell}})= 0$
for all $i, j \in \{1,\dots,n\}$. This implies that the component functions of $g$ satisfy the quasilinear elliptic system
\begin{equation}\label{poiss}
\Delta_g (g_{i\ov{\jmath}})= Q_{i\ov{\jmath}} = \sum g^{k\ov{q}}g^{p\ov{\ell}}\de_i g_{k\ov{\ell}}\de_{\ov{\jmath}}g_{p\ov{q}}.
\end{equation}
For the bootstrap we also require three suitably nested compact regions $K'' \supset K' \supset K$.

By (\ref{metric2}) and (\ref{great}), $\|Q_{i\ov{\jmath}}\|_{L^p(K'',\delta)} \leq C_{K}$ for all $p \geq 1$. Thus,
$\|g_{i\ov{\jmath}}\|_{W^{2,p}(K',\delta)}\leq C_{K,p}$ for $p > 1$ by $L^p$ regularity theory since the coefficient matrix of (\ref{poiss}) has bounded ellipticity and bounded modulus
of continuity by (\ref{metric2}), (\ref{great}). Then $\|g_{i\ov{\jmath}}\|_{C^{1,\alpha}(K',\delta)}\leq C_{K,\alpha}$ for all $\alpha \in (0,1)$ by Morrey's inequality.

Now whenever $g$ is bounded in $C^{k,\alpha}(K'',\delta)$ for some $k \geq 1$, then $Q$ is bounded in $C^{k-1,\alpha}(K'',\delta)$, so that $g$ is bounded in $C^{k+1,\alpha}(K',\delta)$ by interior Schauder theory, which can be used here because the coefficients of (\ref{poiss}) are trivially bounded in $C^{k-1,\alpha}(K'',\delta)$ by assumption. Thus, shrinking and relabeling $K''$ and $K'$ in each step, we inductively prove that $\|g\|_{C^{k}(K,\delta)} \leq C_{K,k}$ for all $k \geq 2$. Again these estimates imply (\ref{main}) for the corresponding values of $k$ because $t \leq 1$.

This completes the proof of Theorem \ref{main}.

 \end{document}